\documentclass[12pt]{amsart}
\usepackage{amsfonts, amscd,color,amsmath,amssymb}
\usepackage{tikz}
\usepackage{graphics}
\usepackage{epic,eepic}

\newtheorem{theorem}{Theorem}[section]

\newtheorem{proposition}[theorem]{Proposition}

\theoremstyle{definition}
\newtheorem{definition}[theorem]{Definition}
\newtheorem{example}[theorem]{Example}

\theoremstyle{remark}
\newtheorem{property}[theorem]{Property}

\numberwithin{equation}{section}

\begin{document}
\title[Wilder continua]{Wilder continua and their subfamilies as coanalytic absorbers}

\author[K. Kr\'olicki]{Konrad Kr\'olicki}
\email{konee.0@gmail.com}
\author[P. Krupski]{Pawe\l\ Krupski}
\email{pawel.krupski@uwr.edu.pl}
\address{Mathematical Institute, University of Wroc\l aw, pl.
Grunwaldzki 2/4, 50--384 Wroc\l aw, Poland}
\date{\today}

\begin{abstract} The  family of  Wilder continua in the cube $I^n$ and its two subfamilies---of continuum-wise Wilder continua and of hereditarily arcwise connected continua---are recognized as coanalytic absorbers in the hyperspace $C(I^n)$ of subcontinua of $I^n$ for $3\leq n\leq\infty$. In particular, each of them is homeomorphic to the set  of all nonempty countable closed subsets of the unit interval $I$.
\end{abstract}

\maketitle

\section{Introduction}

By a continuum we mean a nonempty Hausdorff compact connected space.

\begin{definition}\cite{Wilder1}\label{D1}
A topological space X has the \emph{Wilder property} if it has
at least three points and for any mutually distinct points $x, y, z \in X$  there
exists a continuum $K \subset X$ containing $x$ and exactly one of the points $y, z$.
A continuum that has the Wilder property is called a \emph{Wilder continuum}.
A continuum each of whose nondegenerate
subcontinua has the Wilder property is called a \emph{hereditarily Wilder
continuum.}
\end{definition}

 The  Wilder property
was introduced by B.~E.~Wilder in~\cite{Wilder1} under the name of property $C$. However, in order
to avoid confusion with a much more popular concept of a $C$-space existing in dimension theory (cf.~\cite{E}), we decided
to change the name.  It was shown by Wilder that both arcwise connected  spaces (by an arc we mean a continuum with exactly two non-separating points)
and  aposyndetic  continua have the Wilder property and  plane examples were given of Wilder continua that are neither arcwise connected nor aposyndetic~\cite{Wilder1}, ~\cite{Wilder3}.   Moreover, by~\cite[Corollary 2]{Wilder1}, one can observe that  a  nondegenerate continuum
is a hereditarily Wilder continuum if and only if it is hereditarily arcwise connected.
Wilder continua were also studied in~\cite{Wilder2},~\cite{Lon} and~\cite{Ban}.

Besides (hereditarily) Wilder continua, we introduce a narrower class of \emph{continuum-wise Wilder continua} (in this vein, Wilder continua could well be called point-wise Wilder).

\begin{definition}\label{D2}
 A space $X$ has the \emph{continuum-wise Wilder property}, if it
contains at least three points and if for any mutually disjoint subcontinua
$A,B,C$ of $X$ there exists a subcontinuum $L \supset A$ containing exactly one
of $B,C$ and disjoint with the other one (i.e. either $B \subset L$, $C \cap L =\emptyset$  or
vice-versa).
\end{definition}

\begin{example}
A plane Wilder continuum which is not continuum-wise Wilder can be obtained from the set
\begin{multline*}
(\{-1,1\}\times [-1,1])\cup ( [-1,1]\times \{-1,1\})\cup
\\
\{(x,\sin\frac1{x-1}): x\in (1,2]\}\cup \{(x,\sin\frac1{x+1}): x\in (-2,-1]\}
\end{multline*}
by identifying points $(2,\sin 1)$ and $(-2,\sin(-1))$.

\end{example}

Using Wilder's result that each irreducible continuum with his property is an arc~\cite{Wilder1}, one can easily  observe that in the realm of hereditarily unicoherent continua all three types of Wilder continua coincide with dendroids ($\equiv$  nondegenerate hereditarily unicoherent arcwise connected continua).

 \begin{proposition}
If $X$ is a nondegenerate hereditarily unicoherent continuum, then the following statements are equivalent.

\begin{itemize}
\item
$X$ is Wilder,
\item
$X$ is a dendroid,
\item
$X$ is hereditarily Wilder,
\item
 $X$   is  continuum-wise Wilder.
 \end{itemize}
\end{proposition}

Arguments similar to those for Wilder continua show that the class of nondegenerate arcwise connected  continua is properly contained in the class of continuum-wise Wilder continua.
This is not true for aposyndetic continua.

\begin{example}
It is known that the product of arbitrary nondegenerate continua is aposyndetic. Let $X$ be a continuum containing three points $a$, $b$ and $c$ such that $X$ is irreducible between $a$ and $b$ and between $a$ and $c$ and let $Y$ be a nondegenerate continuum. Then $X\times Y$ is aposyndetic but it is not continuum-wise Wilder, since, for subcontinua $A=\{a\}$, $B=\{b\} \times Y$ and $C=\{c\} \times Y$, if $L\subset X\times Y$ is a continuum which contains $A\cup B$ then the projection $L_X$ of $L$ into $X$ contains $\{a,b\}$, so $L_X=X$ and $L\cap C\neq\emptyset$.
\end{example}

\

Henceforth, we will consider only separable metric spaces. The hyperspace $2^X$ of all nonempty compact subsets of a space $X$ is equipped  with the Hausdorff metric and $C(X)$ denotes its subspace consisting of all subcontinua of $X$. Throughout the paper $I=[0,1]$ and cubes $I^n$ are endowed with the Euclidean metrics.

\

 While the theory of absorbing sets is well established (\cite{B} and ~\cite{M} are good references), only a few examples of  coanalytic absorbers in hyperspaces of cubes  have been recognized. They include:
 \begin{itemize}
 \item
 the Hurewicz set of all nonempty closed countable subsets of $I$  in $2^I$~\cite{C1},
 \item
 the set of all hereditarily decomposable subcontinua of $I^n$ in $C(I^n)$, $3\le n\le \infty$~\cite{Sam1},
\item
the set of all strongly countable-dimensional subcontinua of $I^\infty$ of
dimension $\ge 2$ in $C(I^\infty)$~\cite{KS},
\item
the sets of all weakly infinite-dimensional subcontinua of $I^\infty$ of dimension $\ge 2$ and of $C$-subcontinua of $I^\infty$ of dimension $\ge 2$ in $C(I^\infty)$~\cite{Krup}.
\end{itemize}

In this paper we consider the sets
$$\mathcal W\varsupsetneq \mathcal{CW}\varsupsetneq\mathcal{HA}$$  of Wilder continua, continuum-wise Wil\-der continua and nondegenerate hereditarily arcwise connected continua, resp., in  $I^n$ for $2\leq n\leq\infty$ as the subspaces of  $C(I^n)\overset{\text{top}}{=} I^\infty$.  All of them are  characterized in the next section as coanalytic absorbers in $C(I^n)$ if $n\ge 3$. It means, in particular, that they are homeomorphic to the Hurewicz set.

The characterization of $\mathcal{HA}$ seems to be of particular interest because  closely related families such as arcwise connected subcontinua, hereditarily locally connected subcontinua or dendroids so far escape known methods of recognizing  absorbers. The sets of hereditarily locally connected subcontinua and of  dendroids are coanalytic complete in $C(I^\infty)$ (see~\cite{DM},~\cite{CDM}), so they are natural candidates. The class $\mathcal{AC}$ of nondegenerate arcwise connected continua in $I^n$  is $\Pi^1_2$-complete for $n\ge 3$ (see~\cite[Theorem 37.11]{Ke}). It was claimed in~\cite{C2} that $\mathcal{AC}$ is a $\Pi^1_2$-absorber in $C(I^n)$ but, seemingly, the proof was never published. We observe that $\mathcal{AC}$ is strongly coanalytic-universal for $n\ge 2$ and is covered by a $\sigma Z$-set in $C(I^n)$ if $n\ge 3$.

\

For notions undefined in this paper   the reader is referred to the standard books~\cite{Ke},~\cite{M} and~\cite{Nad}.

\section{Main results}

Let $X$ be a topological Hilbert cube with a metric $d$. Recall that a closed subset $A\subset X$ is
called  a $Z$-{\it set} in $X$ if
 for any $\epsilon>0$ there exists a continuous mapping $f:X\to X$
such that  $f(X)\cap A= \emptyset$
and  $\widetilde{d}(f,\operatorname{id}_X):=\sup\{d(f(x),x):x\in X\}<\epsilon$. A countable union of $Z$-sets in $X$ is called a $\sigma Z$-{\it set} in $X$.

A subset $A$  of  $X$ is  \emph{strongly coanalytic-universal}
if for each coanalytic subset
$M$ of the Hilbert cube $I^\infty$  and
each  compact set $K\subset I^\infty$, any embedding $f:I^\infty \to X$ such that
$f(K)$ is a $Z$-set in $X$
can be approximated arbitrarily closely (in the sense of the uniform convergence)
 by an embedding $g:I^\infty\to X$ such that
   $g(I^\infty)$ is a $Z$-set in $X$,
$g|K=f|K$ and $g^{-1}(A)\setminus K=M\setminus K$.

Finally, $A$ is a
\emph{coanalytic absorber} in $X$ provided that
\begin{enumerate}
\item $A$ is a coanalytic set,
\item $A$ is contained in a $\sigma Z$-set in $X$,
\item $A$ is strongly coanalytic-universal.
\end{enumerate}

All coanalytic absorbers are mutually homeomorphic. Moreover,  if $A$ and $B$ are coanalytic absorbers in a Hilbert
cube $X$,  then there is a
homeomorphism $h: X \to X$ such that $h(A)=B$ which is
arbitrarily close (in the sense of metric $\widetilde{d}$) to the identity~\cite[Theorem 5.5.2]{M}.

\begin{proposition}\label{prop1}
If $Y$ is a compact space, then the families $\mathcal W$,  $\mathcal{CW}$ and $\mathcal{HA}$  in $Y$ are coanalytic subsets of the hyperspace $C(Y)$.
\end{proposition}

\begin{proof}
A direct evaluation of the descriptive complexity of the formula:
\begin{multline*}
W\in \mathcal W \qquad \text{iff}\\
|W|>1 \quad \wedge \quad \forall \  x,y,z\in W \ \exists \ K\in C(Y) \\
 K\subset W \wedge x\in K \ \wedge \
\bigl((y\in K \ \wedge z\notin K)\ \vee \ (y\notin K\ \wedge\ z\in K)\bigr)
\end{multline*}
\newline
immediately gives  coanalyticity of $\mathcal W$. Similar evaluations work for two other families.
\end{proof}

\begin{proposition}\label{prop2}
The families $\mathcal W$, $\mathcal{CW}$, $\mathcal{AC}$ and $\mathcal{HA}$  in $I^n$ are contained in the family $\mathcal D(I^n)$ of all decomposable subcontinua of  $I^n$ which is a  $\sigma Z$-set in  $C(I^n)$ if $n\ge 3$.
\end{proposition}

\begin{proof}
Clearly, all the families  are contained in $\mathcal W$ and each  Wilder continuum is decomposable. Moreover,   $\mathcal D(I^n)$ is a $\sigma Z$-set in $C(I^n)$ for $n\ge 3$ by~\cite[Corollary 4.4]{Sam1}.
\end{proof}

\begin{proposition}\label{prop3}
The families $\mathcal W$, $\mathcal{CW}$,  $\mathcal{AC}$ and  $\mathcal{HA}$ in $I^n$ are  strongly coanalytic-universal in $C(I^n)$ if $n\ge 2$.
\end{proposition}
\begin{proof}
For simplicity, we assume that $n<\infty$ but one can easily adapt the proof to the case of the Hilbert cube $I^\infty$. Fix an arbitrary coanalytic set $M\subset I^\infty$, a closed $K\subset I^\infty$, an embedding  $f:I^\infty \to C(I^n)$ such that $f(K)$ is a $Z$-set and $\epsilon>0$. We are going to define a $Z$-embedding which agrees with $f$ on $K$, is $\epsilon$-close to $f$ and satisfies
\begin{multline}\label{eq:uni}
g^{-1}(\mathcal W )\setminus K= g^{-1}(\mathcal {CW} )\setminus K =  g^{-1}(\mathcal{AC})\setminus K =g^{-1}(\mathcal {HA} )\setminus K= \\  M\setminus K.
\end{multline}
We sketch main steps of a construction of $g$  following ideas from~\cite{GM} and~\cite{Sam1}.
First we associate with $M$  a continuous map $\xi: I^\infty \to C(I^n)$  as in~\cite[Lemma 3.4]{Sam1}.
It has the following property:
\begin{property}
For each $q\in I^\infty$,  $\xi(q)$ is a hereditarily unicoherent 1-dimensional continuum such that
$$\bigl((I\times \{0\})\cup (\{0\}\times I)\bigr)\times \{(\underbrace{0,\dots, 0}_{n-2})\}\subset \xi(q)\subset I^2\times \{(\underbrace{0,\dots, 0}_{n-2})\}$$ and  if $q\in M$, then $\xi(q)$ is a dendroid  which is a union of countably many arcs emanating from the segment $\{0\}\times I \times \{(\underbrace{0,\dots, 0}_{n-2})\}$ while if   $q\notin M$, then $\xi(q)$ contains a pseudoarc meeting  $\{0\}\times I \times \{(\underbrace{0,\dots, 0}_{n-2})\}$ at a single point.
\end{property}

Next, we modify map $\xi(q)$  by aggregating countably many copies of  $\xi(q)$ in  $I^2\times \{(\underbrace{0,\dots, 0}_{n-2})\}$:
\begin{equation}\label{def:Xi}
\Xi(q)= \bigcup_i\alpha_i(\xi(q))\cup (I\times\{1\}\times \{(\underbrace{0,\dots, 0}_{n-2})\}),
\end{equation}
where $$\alpha_i(x_1,x_2,0,\dots,0)=(x_1, \frac{1}{i(i+1)}x_2+\frac{i-1}{i},0,\dots,0), \quad i=1,2,\dots $$
(Figure~\ref{fig1}).

\begin{figure}[h]
\setlength{\unitlength}{1mm}
\begin{picture}(40,50)
\thicklines

\drawline(0,0)(40,0)
\drawline(0,0)(0,40)
\drawline(0,40)(40,40)
\drawline(0,20)(40,20)
\drawline(0,26.6)(40,26.6)
\drawline(0,30)(40,30)
\drawline(0,32)(40,32)
\put(-6,-3){\makebox(0,0){$(0,0,0,\dots)$}}
\put(46,43){\makebox(0,0){$(1,1,0,\dots)$}}
\put(17,9){\makebox(0,0){$\alpha_1(\xi(q))$}}

\put(17,23){\makebox(0,0){$\alpha_2(\xi(q))$}}
\put(20,36.5){\makebox(0,0){$\vdots$}}

\end{picture}
\caption{$\Xi(q)$}\label{fig1}
\end{figure}

Hence, for $q\in M$, $\Xi(q)\in  \mathcal {HA}  \subset \mathcal{AC}\subset \mathcal{CW}\subset \mathcal W$. For $q\notin M$, $\Xi(q)\notin \mathcal W$ because of the subsequent property.
\begin{property}\label{property1}
If $P\subset \Xi(q)$ is a pseudoarc and $a,b,c\in P$ are points such that $P$ is irreducible between any two of them, then for any continuum $L\subset \Xi(q)$ containing $a,b$, the continuum $L\cap P$ must contain $c$.
\end{property}
This shows that
\begin{equation}\label{eq1}
\Xi^{-1}(\mathcal W)=  \Xi^{-1}(\mathcal{CW})= \Xi^{-1}(\mathcal {AC})=\Xi^{-1}(\mathcal {HA})=M.
\end{equation}

Apart from  $\Xi$, we need another map $\theta:I^\infty \to C([-1,1]^n)$
sending $q=(q_i)$ to
$$\theta(q)=  \bigl(([-1,0]\times \{0\})\cup S((-\frac12,0);\frac12)\cup \bigcup_{i=1}^\infty S(a_i;r_i(q))\bigr)\times \{(\underbrace{0,\dots, 0}_{n-2})\},$$
where $S(x;r)$ denotes the circle in the plane centered at $x$ with radius $r$, $a_i=(-1+2^{-i},0)\in \mathbb R^2$ and $r_i(q)=4^{-(i+1)}(1+q_i).$

The figure $\theta(q)$ is the union of countably many disjoint circles and of the diameter segment of the largest circle lying in $[-1,0]\times [-1,1]\times \{(\underbrace{0,\dots, 0}_{n-2})\}$. The map $\theta$ is a continuous embedding.

We also use two deformations. The first one  $H_0:2^{I^n}\times I\to 2^{I^n}$ is such that, for any $(A,t)\in 2^{I^n}\times (0,\frac12]$,  $H_0(A,t)$ is finite, $dist(A,H_0(A,t))\le 2t$ ($dist$ is the Hausdorff distance) and $H_0(A,t)\subset [t,1-t]^n$. The second $H:C(I^n)\times [0,\frac12]\to C(I^n)$ is defined by
$$H(A,t) =  H_0(A,t)\cup \bigcup_{a,b\in H_0(A,t)} (\overline{ab}\cap (\overline{B}(a; 2t)\cup  \overline{B}(b; 2t)))$$
where $\overline{B}(a; \alpha)$ is the closed $\alpha$-ball in $I^n$ around $a$ and $\overline{ab}$ is the line segment in $I^n$ from $a$ to $b$.

It is known that for any  $(A,t)\in C(I^n)\times (0,\frac12)$,  $H(A,t)$ is a connected graph in $[t,1-t]^n$  and $dist(A,H(A,t))\le 4t$  (see~\cite{GM} and~\cite{Sam1} for details and other properties of $H$).

For each $q\in I^\infty$, put $$\mu(q) =\frac1{12}\min\{\epsilon, \min\{dist(f(q),f(q')): q'\in f(K)\}\}$$ and define our approximation
\begin{multline}\label{eq:g}
g(q)=  H(f(q),\mu(q))\  \cup \\
\bigcup_{x\in H_0(f(q),\mu(q))}(x+\mu(q)\theta(q))\  \cup \   \bigcup_{x\in H_0(f(q),\mu(q))}(x+\mu(q)\Xi(q))
\end{multline}
(we use linear operations of addition and scalar multiplication in~\eqref{eq:g}).
It is shown in the proof of~\cite[Lemma 3.2]{Sam1} that $g$ (denoted there by $G$; the definition of $G$ in~\cite{Sam1}  should be corrected to the effect as in~\eqref{eq:g})
 is a $Z$-embedding which is  $\epsilon$-close to $f$ and coincides with $f$ on $K$.

In order to see that~\eqref{eq:uni} is satisfied, suppose that $q\notin K$. Then $\mu(q)>0$. If $q\in M$, then one can easily see that $g(q)\in \mathcal {HA}$.

Assume now that $q\notin M$.  It is convenient to consider the vertex $$v=(v_1,v_2,v_3,\dots,v_n)\in H_0(f(q),\mu(q))$$ which is  maximal  in $H_0(f(q),\mu(q))$ with respect to the lexicographic order $\prec$ on $I^n$. Observe that the copy $v+\mu(q)\Xi(q)$ of $\Xi(q)$ can be intersected only by finitely many other copies  $v'+\mu(q)\Xi(q)$ for $v'=(v'_1,v'_2,v_3,\dots,v_n) \prec v$. Also, $v+\mu(q)\Xi(q)\setminus \{v\}$ is disjoint from $$H(f(q),\mu(q))\  \cup
\bigcup_{x\in H_0(f(q),\mu(q))}(x+\mu(q)\theta(q)).$$ It follows that there exists an open in $g(q)$ neighborhood $U$ of the point $v+\mu(q)(1,1,0,\dots,0)$ such that $$\overline{U}\subset v+\mu(q)\bigl(\Xi(q)\setminus \{0\}\times I\times \{(\underbrace{0,\dots, 0}_{n-2})\}\bigr).$$
By  definition~\eqref{def:Xi} of $\Xi$ and since $q\notin M$, some component of $\overline{U}$ is a pseudoarc $P$. Take an open subset $V$ of $U$ such that $V\cap P\neq \emptyset$ and $\overline{V}\subset U$. Let $P'$ be a component of $\overline{V}$ contained in $P$, $a,b,c\in P'$ be points from distinct composants of $P'$ and let
$L\subset g(q)$ be a continuum containing $a,b$. If $L\subset  \overline{U}$, then $ L\subset P$, hence $P'\subset L$  by the hereditary indecomposablity of $P$. In case $L\setminus \overline{U}\neq\emptyset$, the component $L_a$ of $L\cap \overline{U}$ containing point $a$ meets the boundary of $U$ by the Janiszewski Boundary Bumping Theorem~\cite[5.4]{Nad}, so $L_a\cap P'\neq\emptyset\neq L_a\setminus P'$ and, of course, $L_a\subset P$. Thus, $P'\subset L_a$ as $P$ is hereditarily indecomposable. In any case, we get $c\in L$.

Consequently, $g(q)\notin \mathcal {W}$.
\end{proof}

Our main theorem follows from Propositions~\ref{prop1},~\ref{prop2},~\ref{prop3}.

\begin{theorem}\label{th1}
The families  $\mathcal{HA}$ and $\mathcal{CW}$, $\mathcal W$ in $I^n$ are coanalytic absorbers in the hyperspace $C(I^n)$ if $n\ge 3$.
\end{theorem}

We do not know if Theorem~\ref{th1} remains true for $n=2$.

\bibliographystyle{amsplain}

\end{document}